\documentclass{amsart}
\usepackage{latexsym,amsxtra,amscd,ifthen,amsmath,color}
\usepackage{amsfonts}
\usepackage{verbatim}
\usepackage{amsmath}
\usepackage{amsthm}
\usepackage{amssymb}
\usepackage[notcite,notref,final]{showkeys}

\numberwithin{equation}{section}

\theoremstyle{plain}
\newtheorem{theorem}{Theorem}[section]
\newtheorem{lemma}[theorem]{Lemma}

\newtheorem{proposition}[theorem]{Proposition}
\newtheorem{hypothesis}[theorem]{Hypothesis}

\newtheorem{definitionlemma}[theorem]{Definition-Lemma}

\theoremstyle{definition}
\newtheorem{definition}[theorem]{Definition}

\newtheorem{remark}[theorem]{Remark}
\newtheorem{question}[theorem]{Question}

\makeatletter              
\let\c@equation\c@theorem  
\makeatother

\DeclareMathOperator{\hdet}{hdet}

\DeclareMathOperator{\hcodet}{hcodet}

\DeclareMathOperator{\gldim}{gldim}

\DeclareMathOperator{\uend}{\underline{end}}

\DeclareMathOperator{\Ext}{Ext}

\DeclareMathOperator{\GKdim}{GKdim}

\DeclareMathOperator{\eEnd}{end}

\DeclareMathOperator{\D}{{\sf D}}
\DeclareMathOperator{\Hom}{Hom}

\DeclareMathOperator{\GrMod}{{\sf GrMod}}

\DeclareMathOperator{\Per}{{\sf Per}}

\newcommand{\fe}{\mathfrak{e}}

\newcommand{\mc}{\mathcal}

\newcommand{\ch}{\operatorname{char}}

\begin{document}

\title{Hopf actions and Nakayama automorphisms}

\author{Kenneth Chan, Chelsea Walton, and James Zhang}

\address{Chan: Department of Mathematics, Box 354350,
University of Washington, Seattle, Washington 98195,
USA}

\email{kenhchan@math.washington.edu}

\address{Walton: Department of Mathematics, Massachusetts
Institute of Technology, Cambridge, Massachusetts 02139,
USA}

\email{notlaw@math.mit.edu}

\address{zhang: Department of Mathematics, Box 354350,
University of Washington, Seattle, Washington 98195,
USA}

\email{zhang@math.washington.edu}



\begin{abstract}
Let $H$ be a Hopf algebra with antipode $S$, and let $A$ be an
$N$-Koszul Artin-Schelter regular algebra. We study connections
between the Nakayama automorphism of $A$ and $S^2$ of $H$ when $H$
coacts on $A$ inner-faithfully. Several applications pertaining to
Hopf actions on Artin-Schelter regular algebras are given.
\end{abstract}

\maketitle

\bibliographystyle{alpha}

\setcounter{section}{-1}
\section{Introduction}
\label{sec0}

This article is a study in noncommutative invariant theory,
particularly on the actions of finite dimensional Hopf algebras on
Artin-Schelter (AS) regular algebras.  Our results also lay the groundwork for other studies on Hopf actions on (filtered) AS regular algebras, namely for both \cite{CKWZ} and \cite{CWWZ:filtered}. To begin, we discuss the vital role of Nakayama automorphisms.

Let $k$ be a base field and let $B$ be either a connected graded
AS regular algebra or a noetherian AS regular Hopf algebra.  An
algebra automorphism $\mu_B$ of $B$ is called a
{\it Nakayama automorphism} of $B$ if there is an integer
$d\geq 0$ such that
\begin{equation}
\label{E0.0.1}\tag{E0.0.1} \Ext^i_{B^e}(B,B^{e})\cong \begin{cases}
{^{\mu_B} B^1}& {\text{if}} \quad i=d,
\\ 0& {\text{if}} \quad i\neq d\end{cases}
\end{equation}
as $B$-bimodules, where  $B^e=B\otimes B^{op}$
\cite[Definition 4.4(b)]{BrownZhang:Dualizing}. The algebra $B$ is
called {\it Calabi-Yau} if $\mu_B=Id$. Also, the quantity $d$ is the
global dimension of $B$ when $B$ is as given above. The definition of
$\mu_B$ is motivated by the classical notion of the  Nakayama
automorphism of a Frobenius algebra; see  Section
~\ref{secxx1} for details. The Nakayama automorphism $\mu_B$ is also unique up to inner
automorphism of $B$. Further, if $B$ is connected graded, then the Nakayama
automorphism can be chosen to be a graded algebra automorphism, and
in this case, it is unique since $B$ has no non-trivial graded inner
automorphism.

Fairly recently, Brown and third-named author proved that the
Nakayama automorphism of a noetherian AS regular Hopf
algebra $K$  can be written as follows:
\begin{equation}
\label{E0.0.2}\tag{E0.0.2}\mu_K=S^2\circ \Xi^l_{\int^l}.
\end{equation}
Here, $S$ is the antipode of $K$ and $\Xi^l_{\int^l}$ is the left winding automorphism of $K$ associated to the left homological integral $\int^l$ of $K$
\cite[Theorem 0.3]{BrownZhang:Dualizing}. This illustrates
how one can express homological invariants (e.g., the Nakayama
automorphism) in terms of other invariants (e.g., $S^2$ and
$\int^l$) of such an AS regular Hopf algebra $K$.

Now in this paper, we consider (co)actions of a Hopf algebra $K$ on connected graded AS regular algebras $A$ and we prove a result in the same vein of (\ref{E0.0.2}); c.f. Theorem \ref{thmxx0.1}. This establishes a connection between properties (as
well as numerical invariants) of such $A$ and $K$. One of the key
ideas is to use Manin's construction of quantum linear groups
\cite{Manin:QGNCG}\cite{Manin:Topics} to express $S^2$ of $K$ in terms
of the Nakayama automorphism of $A$. As a result, for large class of 
AS regular algebras $A$, we describe the structure of the Hopf 
algebras $H=K^*$ that act on $A$; see Theorems \ref{thmxx0.4}, 
\ref{thmxx0.6} and Proposition \ref{proxx0.7} below.

Before we state our main result, we introduce inner-faithful Hopf
actions. Let $N$ be a right $K$-comodule via the comodule structure
map $\rho: N \rightarrow N \otimes K$. In other words, $K$ coacts on
$N$. We say that this coaction is {\it inner-faithful} if for any
proper Hopf subalgebra $K'\subsetneq K$, we have that
$\rho(N)\not\subset N\otimes K'$. A left $K$-module $M$ is called
{\it inner-faithful} if there is no nonzero Hopf ideal $I\subset K$
such that $IM=0$.

\begin{theorem}
\label{thmxx0.1} Let $A$ be a connected graded $N$-Koszul AS regular
algebra with Nakayama automorphism $\mu_A$. Here, $N \geq 2$. Let
$K$ be a Hopf algebra with bijective antipode $S$ coacting on $A$
from the right. Suppose that the homological codeterminant
(Definition \ref{defxx1.6}(b)) of the $K$-coaction on $A$ is the
element ${\sf D} \in K$ and that the $K$-coaction on $A$ is
inner-faithful. Then
\begin{equation}
\label{E0.1.1}\tag{E0.1.1} \eta_{\D} \circ S^2=\eta_{\mu_A^\tau},
\end{equation}
where $\eta_{\sf D}$ is the automorphism of $K$ defined by
conjugating ${\sf D}$ and $\eta_{\mu_A^\tau}$ is the automorphism
of $K$ given by conjugating by  the transpose of the corresponding
matrix of $\mu_A$.
\end{theorem}

The automorphism on the left-hand side of equation \eqref{E0.1.1}
is the composition of the Hopf algebra automorphism $S^2$ of $K$ 
(which is bijective by hypothesis) and the Hopf algebra automorphism 
$\eta_D$ of $K$ where $\eta_{\sf D}$ is given by $\eta_{\sf D}(a)
={\sf D}^{-1}a{\sf D}$ for all $a\in K$. The right-hand side of 
equation \eqref{E0.1.1} needs some explanation. First we start
with a $k$-linear basis, say $\{x_1,\cdots,x_n\}$, of $A_1$,
the degree 1 graded piece of $A$. Then the Nakayama automorphism 
$\mu_A$ is $A$ can be written as 
$$\mu_A (x_j)=\sum_{i=1}^n m_{ij} x_j$$
for all $j=1,\cdots,n$. Let $M$ be the $n\times n$-matrix $(m_{ij})$
over the base field $k$. Since $K$ coacts on the $\Ext$-algebra
$E:=\Ext^*_A(k,k)$, we have a set of elements $\{y_{ij}\}_{n\times n}$ 
in $K$ such that $\rho(x_i^*)=\sum_{s=1}^n y_{is} \otimes x_s^*$.
Since the $K$-coaction on $A$ is inner-faithful, 
$\{y_{ij}\}_{n\times n}$ generates $K$ as a Hopf algebra. 
With this choice of $\{x_i\}_{i=1}^n$, we define 
$\eta_{\mu_A^\tau}: K\to K$ by
$$\eta_{\mu_A^\tau}: y_{ij}\mapsto \sum_{s,t=1}^n m_{si} y_{st}n_{jt}$$
for all $1\leq i, j\leq n$ where $(n_{ij})_{n\times n}=
(m_{ij})_{n\times n}^{-1}$. Roughly speaking, we use ``coordinates'' 
to define the ``conjugation'' automorphism $\eta_{\mu_A^\tau}$ of $K$. It is 
worth noting that  Theorem \ref{thmxx0.1} implies that the definition 
of this automorphism is independent of the choice of coordinates. 
We also conjecture that Theorem~\ref{thmxx0.1} should hold when $A$
is not necessarily $N$-Koszul.

There are some similarities between the equation
\eqref{E0.0.2} and \eqref{E0.1.1} and it would be very 
interesting these come from a single more general equation.

\begin{question}
\label{quexx0.2} Is there a way of unifying \eqref{E0.0.2} and
\eqref{E0.1.1}?
\end{question}

Theorem \ref{thmxx0.1} has several applications to the study of finite
dimensional Hopf actions on connected graded AS regular algebras.
For the rest of the introduction, we consider Hopf algebra actions
(instead of coactions). Further, we impose the following hypotheses
for the rest of the article unless stated otherwise.

\begin{hypothesis}
\label{hypxx0.3} We assume that
\begin{enumerate}
\item[(i)]
$H$ is a finite dimensional Hopf algebra,
\item[(ii)]
$A$ is a connected graded AS regular algebra,
\item[(iii)]
$H$ acts on $A$ inner-faithfully, and
\item[(iv)]
the $H$-action on $A$ preserves the grading of $A$.
\end{enumerate}
\end{hypothesis}

The first consequence of Theorem \ref{thmxx0.1} is the following result.

\begin{theorem}
\label{thmxx0.4} Let $k$ be an algebraically closed field.  Let $H$
act on a skew polynomial ring $A=k_p[x_1,\cdots,x_n]$ where $p$ is
not a root of unity, then $H$ is a group algebra.
\end{theorem}

More generally, a version of Theorem \ref{thmxx0.4} holds for a large class of multi-parameter skew
polynomial rings; see Theorem \ref{thmxx4.3}.

Since $H$ ends up being a group algebra in the theorem above, we
consider this Hopf action to be {\it trivial}, that is to say, $H$ is either
a commutative or a cocommutative Hopf algebra. On the other hand, there are many
non-trivial finite dimensional Hopf algebra actions on skew
polynomial rings $k_p[x_1,\cdots,x_n]$ where $p$ is a root of unity.
For instance, some interesting non-trivial Hopf algebra actions on
$k_p[x_1,x_2]$ are given in \cite{CKWZ}. Thus, Theorem \ref{thmxx0.4}
prompts the question below.

\begin{question}
\label{quexx0.5}
Suppose that $H$ acts on $A$ under the assumptions of Hypothesis
\ref{hypxx0.3}. If $A$ does not contain the commutative polynomial
ring $k[x_1,x_2]$ as a subalgebra, is then $H$ a group algebra?
\end{question}

Theorem \ref{thmxx5.10} provides positive  evidence for Question
\ref{quexx0.5}. Another consequence of Theorem \ref{thmxx0.1} is
following result.

\begin{theorem}
\label{thmxx0.6} Suppose that $\ch k$ does not divide the dimension of $H$. Assume also that $A$ is
$N$-Koszul and so-called $r$-Nakayama. If the $H$-action on $A$ has
trivial homological determinant (Definition~\ref{defxx1.6}(b)), then $H$ is semisimple.
\end{theorem}

Examples of $r$-Nakayama algebras include the Sklyanin algebras of
dimension~3 and 4, commutative polynomial rings $k[x_1,\cdots,x_n]$,
and the skew polynomial ring $k_{-1}[x_1,\cdots,x_n]$. 
Hence, Theorem
\ref{thmxx0.6} applies to such algebras. Moreover, if $\ch k =0$, then as a consequence of \cite[Theorem 0.1]{KKZ:Gorenstein},
 a pair $(H,A)$ that satisfies the hypotheses of Theorem
\ref{thmxx0.6} yields Artin-Schelter Gorenstein invariant subring $A^H$.

In the case when $H$ is semisimple, we show that there are only
trivial Hopf algebra actions on the commutative polynomial ring of
two variables. Note that this is a special case of \cite[Theorem~1.3]{EtingofWalton:ssHopf}.

\begin{proposition}
\label{proxx0.7}
Suppose $H$ is semisimple and acts on the commutative polynomial
ring $A=k[x_1,x_2]$, where $k$ is an algebraically closed field of
characteristic 0.  Then $H$ is a group algebra.
\end{proposition}

In summary, if $H$ acts on $A$ inner-faithfully,
then algebraic properties
of $A$ affect the structure of $H$. We also conjecture that the converse relationship holds.

\medskip

This article is organized as follows. In Section~\ref{secxx1}, we provide
background material on Artin-Schelter regularity, Nakayama automorphisms, Hopf algebra actions, and other topics. Section~\ref{secxx2}  introduces several quantum groups associated to
a graded algebra, which first appeared in \cite{Manin:QGNCG}. We
provide preliminary results on Hopf actions on Frobenius algebras in
Section~\ref{secxx3}. Moreover, we prove the main results, namely Theorems
\ref{thmxx0.1}, \ref{thmxx0.4}, and \ref{thmxx0.6}, in Section~\ref{secxx4}. We
also discuss $N$-Koszul algebras in Section~\ref{secxx4}. Finally in Section~\ref{secxx5},
we provide examples of the main theorems and the quantum groups of Section~\ref{secxx2} by considering actions on AS regular algebras
of global dimension 2. Namely, see Propositions~\ref{proxx5.4},~\ref{proxx5.5},~\ref{proxx5.9} and Theorem~\ref{thmxx5.10}. We also prove Proposition \ref{proxx0.7} in
Section~\ref{secxx5}.

All vector spaces, algebras, and rings are over the base field $k$.
The unmarked tensor $\otimes$ means $\otimes_k$. The $k$-linear dual
of a vector space $V$ is denoted by $V^*$. In general, we use $H$
(respectively, $K$) to denote a Hopf algebra that acts
(respectively, coacts) on an algebra $A$.

\section{Background Material}
\label{secxx1}

This section
contains some preliminary material that is needed for this article.
We refer to \cite{Montgomery} for basic definitions regarding Hopf
algebras.

\subsection{Artin-Schelter regularity}

This article focuses on the actions of Hopf algebras on certain algebras; such algebras are given as follows.

\begin{definition}
\label{defxx1.1}
Let $A$ be a locally finite, connected, $\mathbb{N}$-graded $k$-algebra. Namely, $A = \bigoplus_{i \geq 0} A_i$ with $A_i A_j \subset A_{i+j}$, $A_0=k$, and $\dim_k A_i < \infty$. We say that $A$ is
{\it Artin-Schelter Gorenstein} (or {\it AS Gorenstein}) if
\begin{enumerate}
\item[(i)] $A$ has finite injective dimension $d < \infty$, and
\item[(ii)] $\Ext^i_A(_Ak,_AA) \cong \Ext^i_A(k_A,A_A)=
\begin{cases} 0& i\neq d\\
k(l) &i=d\end{cases}\quad$ for some
$l \in \mathbb{Z}$.

\noindent
We call $l$ the {\it AS index} of $A$.
\end{enumerate}
If further
\begin{enumerate}
\item[(iii)] $A$ has finite global dimension,
\end{enumerate}
then $A$ is called {\it Artin-Schelter regular} (or {\it AS regular}).
\end{definition}

In this paper, we are not assuming that $A$ has finite Gelfand-Kirillov
dimension as in the standard definition of AS regularity.

\subsection{Frobenius algebras and classical Nakayama automorphisms}

Next we recall the definition of a Frobenius algebra. Let $(G,+)$
be an abelian group (such as ${\mathbb Z}^d$ for $d\geq 0$).
A $G$-graded, finite dimensional, unital, associative algebra $B$ is
called {\it Frobenius} if there is a
nondegenerate associative bilinear form
$$\langle-,-\rangle: B\times B\to k$$
which is graded of degree $-l$ for some $l\in G$. If
$G$ is $\{0\}$, then this is the classical definition of a
Frobenius algebra.

If $B$ is a connected graded algebra, then $B$ is a Frobenius
algebra if and only if $B$ is AS Gorenstein of injective dimension
zero. We call $l$ the {\it AS index} of $B$, which agrees with the
AS index defined in Definition \ref{defxx1.1}(ii) when $B$ is
connected graded.  Moreover, $B$ is Frobenius if
the $k$-linear dual $B^*$ is isomorphic to $B(-l)$ as graded left
$B$-modules. A nice discussion about Frobenius algebras can be
found in \cite{Smith:elliptic}.

The {\it classical Nakayama automorphism} of  a Frobenius algebra
$B$ is a $G$-graded algebra automorphism $\mu_B$ of $B$ satisfying
$$\langle a,b\rangle=\langle b, \mu_B(a)\rangle$$
for all $a,b\in B$. It is well-known that the $k$-linear dual $B^*$
is isomorphic to ${^{\mu_B} B^1}(-l)$ as $B$-bimodules.

\subsection{Hopf algebra actions}

If $H$ is a Hopf algebra, we denote by $H^\circ$  its Hopf dual as
in \cite[Theorem 9.1.3]{Montgomery}. If $H$ is finite dimensional,
then $H^\circ=H^*$ as a $k$-vector space. We say that $H$ {\it
(co)acts} on an algebra $A$ if $A$ arises as an $H$-{\it (co)module
algebra}; refer to \cite[Definitions 4.1.1 and 4.1.2]{Montgomery}
for the definition of a $H$-module algebra and a $H$-comodule
algebra. Note that  if $H$ is finite dimensional, then $M$ is a left
$H$-module if and only if $M$ is a right $H^{\circ}$-comodule. We
also provide the definition of inner-faithfulness.

\begin{definition}
\label{defxx1.2}
Let $M$ be a left $H$-module. We say that the $H$-action on $M$ is
{\it inner-faithful} \cite[Definition 2.7]{BanicaBichon} if $IM\neq 0$
for any nonzero Hopf ideal $I\subset H$. Let $N$ be a right $H$-comodule
with comodule structure map $\rho: N \rightarrow N \otimes H$. We say
that this coaction is {\it inner-faithful} if $\rho(N)\not\subset
N\otimes H'$ for any proper Hopf subalgebra $H'\subsetneq H$.
\end{definition}

Here is an easy lemma about inner-faithfulness and its proof is omitted.

\begin{lemma}
\label{lemxx1.3}
Let $H$ be a Hopf algebra and $K:=H^{\circ}$. Let $M$ be a left
$H$-module, or equivalently, a right $K$-comodule with comodule
structure map $\rho: M \rightarrow M \otimes K$.
\begin{enumerate}
\item
If $H$ is finite dimensional, then the left $H$-action on $M$ is
inner-faithful if and only if the right $K$-coaction on $M$ is
inner-faithful.
\item
Suppose $\{x_i\}$ is a basis of $M$ and assume that $\rho(x_i)=
\sum_s x_s\otimes a_{si}$ for some elements $a_{si}\in K$. Then, the
right $K$-coaction on $M$ is inner-faithful if and only if the subalgebra
generated by $\{a_{si}, S^n(a_{si})\}_{s,i,n}$ is $K$.
\item
Let $R$ be an algebra generated by the $H$-module $M$. Suppose that
$R$ is an $H$-module algebra with $H$-action induced by the $H$-action on $M$.
Then, the $H$-action on $R$ is inner-faithful if and only if the $H$-action
on $M$ is inner-faithful.  \qed
\end{enumerate}
\end{lemma}

\subsection{Homological (co)determinant of a Hopf (co)action}

We now recall the definition of homological determinant in two
different settings:
\begin{enumerate}
\item[(1)]
The original definition given in \cite{KKZ:Gorenstein} when the Hopf algebra $H$ is finite
dimensional and $A$ is noetherian connected graded AS Gorenstein;
\item[(2)]
$H$ is infinite dimensional and $A$ is (not necessarily
noetherian) AS regular.
\end{enumerate}
Since there is no known uniform definition to cover
 these  cases, we
present below the various definitions of homological (co)determinant.
\smallskip

\noindent  \underline{Case (1):} Let $A$ be a noetherian connected
graded AS Gorenstein algebra with
AS index $l$. Let $d$ be the injective dimension of $A$ and let
$R^d\Gamma_{\mathfrak m}(A)$ be the $d$-th local cohomology of $A$. Then
the $k$-linear dual $R^d\Gamma_{\mathfrak m}(A)^*$ is isomorphic to
the graded $A$-bimodule ${^\mu A^1}(-l)$ where $\mu$ is the Nakayama
automorphism of $A$. Suppose $H$ acts on $A$ such that $A$ is a
left $H$-module algebra. When $H$ is finite dimensional, then
$H$ acts on $R^d\Gamma_{\mathfrak m}(A)^*$ from the right; see
\cite[Section 3]{KKZ:Gorenstein} for more details.
Let ${\mathfrak e}\in R^d\Gamma_{\mathfrak m}(A)^*$ be a basis element
of the lowest degree (i.e., degree $l$). Then, there is an algebra
homomorphism $\gamma: H\to k$ such that
$${\mathfrak e}\cdot h=\gamma(h) {\mathfrak e}$$
for all $h\in H$.

\begin{definition} \label{defxx1.4}
Retain the notation above.
\begin{enumerate}
\item
By \cite[Definition 3.3]{KKZ:Gorenstein}, the
composition $\gamma\circ S: H\to k$ is called the {\it homological
determinant} of the $H$-action on $A$ and it is denoted by $\hdet_H A$.
If $\hdet_H A = \epsilon$, the counit, then we say that the homological
determinant is {\it trivial}.
\item
Let $K(=H^{\circ})$ be a Hopf algebra coacting on $A$ from the right
(respectively, from the left)
via the $K$-comodule map $\rho: A \rightarrow A \otimes K$ (respectively,
$\rho: A\to K\otimes A$). The {\it
homological codeterminant} of the $K$-coaction on $A$, denoted by
$\hcodet_KA$, is defined to be the element ${\sf D} \in K$ such that
$\rho({\mathfrak e}) ={\mathfrak e}\otimes {\sf D}^{-1}$ (respectively, 
$\rho({\mathfrak e}) = {\sf D}\otimes {\mathfrak e}$). Note that
${\sf D}$ and ${\sf D}^{-1}$ are necessarily grouplike (whence
invertible) elements. We say the homological codeterminant is {\it
trivial} if $\hcodet_KA=1_K$.
\end{enumerate}
\end{definition}

\noindent
\underline{Case (2):}
Let $A$ be an AS regular algebra for this case. By \cite[Corollary
D]{LPWZ:Koszul}, a connected graded algebra $A$ is AS regular if and only
if the $\Ext$-algebra $E$ of $A$ is Frobenius. Here,

\begin{equation}
\label{E1.5.1}\tag{E1.5.1} E=\bigoplus_{i\geq 0} \Ext^i_A(_Ak,_Ak).
\end{equation}

\begin{definition}
\label{defxx1.6}
Let $A$ be an AS regular algebra with Frobenius $\Ext$-algebra $E$ as above.
Suppose ${\mathfrak e}$ is a nonzero element in
$\Ext^d_A(_Ak,_Ak)$ where $d=\gldim A$.
\begin{enumerate}
\item
Let $H$ be a Hopf algebra with bijective antipode $S$
acting on $A$ from the left. By
\cite[Lemma 5.9]{KKZ:Gorenstein}, $H$ acts on $E$ from the left. The
{\it homological determinant} of the $H$-action on $A$, denoted by
$\hdet_H A$, is defined to be
$\eta\circ S$ where $\eta: H\to k$ is determined by
$$h\cdot {\mathfrak e}=\eta(h) {\mathfrak e}.$$
We say the homological determinant is {\it trivial} if $\hdet_H A=\epsilon$
where $\epsilon$ is the counit of $H$.
\item
Let $K$ be a Hopf algebra with bijective antipode $S$ coacting on
$A$ from the right (respectively, from the left). By a dual version of \cite[Lemma
5.9]{KKZ:Gorenstein}, or Remark \ref{remxx1.7}(d), $K$ coacts on $E$
from the left (respectively, from the right). The {\it homological codeterminant} of the
$K$-coaction on $A$, denoted by $\hcodet_K A$, is defined to be ${\sf
D}$  where $\rho({\mathfrak e})={\sf D}\otimes {\mathfrak e}$ (respectively, $\rho(\mathfrak{e}) = \mathfrak{e} \otimes \sf{D}^{-1}$) for
some grouplike element ${\sf D} \in K$. We say the homological
codeterminant is {\it trivial} if $\hcodet_K A=1_K$.
\end{enumerate}
\end{definition}

Now we make some remarks concerning these definitions. If $H$ is
finite dimensional and if $A$ is noetherian, then
Definition~\ref{defxx1.6}(a) agrees with \cite[Definition~3.3]{KKZ:Gorenstein}
by \cite[Lemma~5.10(c)]{KKZ:Gorenstein}. Moreover under this
assumption, Definition~\ref{defxx1.6}(b)
agrees with \cite[Definition~6.2]{KKZ:Gorenstein} by \cite[Remark~6.3 and
Lemma~5.10(c)]{KKZ:Gorenstein}. Definition \ref{defxx1.6}(a) also appears in
\cite[Definition 1.11]{LWZ}.

Note that the term ``homological determinant'' is prompted by group
actions on commutative polynomial rings. More precisely, in the case
that $A$ is a commutative polynomial ring $k[x_1, \dots, x_n]$, and
$H = kG$, for $G$ a finite subgroup of $GL_n(k)$, $\hdet_HA$ becomes
the ordinary determinant $\det_G : G \rightarrow k$ \cite[Remark
3.4(c)]{KKZ:Gorenstein}. Hence, the homological determinant of the
$H$-action on $A$ is trivial in this case if and only if $G \subseteq
SL_n(k)$.

\begin{remark}
\label{remxx1.7}
Let $H$ and $K$ be Hopf algebras.
Let $A$ be a left $H$-module algebra, and $E$ be the
$\Ext$-algebra $\bigoplus_{i\geq 0}\Ext^i_A(_Ak,_Ak)$.
\begin{enumerate}
\item
In general, $E$ is not a left $H$-module algebra.
\item \cite[Lemma 5.9(d)]{KKZ:Gorenstein}
If $H$ has a bijective antipode $S$, then the opposite ring $E^{op}$
of $E$ is naturally a left $H$-module algebra. As a consequence, $E$
is a right $H$-module algebra induced by the left action of $H$ on
$E^{op}$.
\item
The left $K$-coaction in Definition \ref{defxx1.6}(b) is induced by
the right $K$-coaction given in \cite[Definition 6.2]{KKZ:Gorenstein}.
\item
Suppose $A$ is AS regular and $K$ has bijective antipode.
If $A$ is a right $K$-comodule algebra, then $E$ is a left
$K$-comodule algebra.
\end{enumerate}
\end{remark}

We now provide reasoning for Remark \ref{remxx1.7}(d) as follows.
Let $A-\GrMod^K$ be the category of graded left-right $(A,K)$-Hopf
modules. It is a Grothendieck category with enough injective
objects. Let $A-\Per^K$ be the full subcategory of $A-\GrMod^K$
consisting of objects having a finite free resolution. If $I$ is an
injective object in $A-\GrMod^K$, then \cite[Lemma~2.2(3),
Propositions~2.8(2) and~2.9]{CaenepeelGuedenon} imply that
$\Ext^p_A(M,I)=0$, for all $p>0$ and for all $M\in A-\Per^K$. In
other words, every injective object $I$ in $A-\GrMod^K$ is
$\Hom_A(M,-)$-acyclic for any $M$ in $A-\Per^K$. Note that if $A$ is
AS regular, then $k$ is in $A-\Per^K$ \cite[Proposition
3.1.3]{StephensonZhang}.
 As a consequence of acyclicity of $I$, $\Ext^i_A(k,k)$ can be computed by
using an injective resolution of the second $k$ in $A-\GrMod^K$. By
\cite[Lemma 2.2(3)]{CaenepeelGuedenon}, each $\Ext^i_A(k,k)$ is a
right $K$-comodule. Therefore $E$ is a right $K$-comodule. Applying
the antipode $S$, we have that $E$ is naturally a left $K$-comodule.
To show that $E$ is a left $K$-comodule algebra, one needs to repeat
a similar argument in the proof of \cite[Lemma~5.9]{KKZ:Gorenstein},
but the details are omitted here.

\subsection{$r$-Nakayama algebras}

Let $r\in k^\times$. Let $A$ be a connected
${\mathbb N}$-graded algebra. Define a graded algebra automorphism
$\xi_r$ of $A$ by
$$\xi_r(x)=r^{\deg x} \; x$$
for all homogeneous elements $x\in A$.

\begin{definition}
\label{defxx1.8} Let $A$ be a connected graded AS Gorenstein algebra
such that the Nakayama automorphism exists (defined as in
\eqref{E0.0.1}) and let $r\in k^\times$. We say $A$ is {\it
$r$-Nakayama} if the Nakayama automorphism of $A$ is $\xi_r$. 
\end{definition}

Note
that $A$ is {\it Calabi-Yau} if and only if $A$ is both $1$-Nakayama
and AS regular.

\section{Manin's quantum linear groups}
\label{secxx2}

In this section, we recall Manin's construction of the quantum
groups $\mathcal{O}_A(M)$, $\mathcal{O}_A(GL)$, $\mathcal{O}_A(SL)$,
and $\mathcal{O}_A(GL/S^2)$ associated to an algebra $A$
\cite{Manin:QGNCG}. We review this mainly in the case when $A$ is a quadratic
algebra, but we also remark that the same ideas apply to non-quadratic
algebras at the end of the section. Note that examples of this material
 are provided in Section~\ref{secxx5}.

\subsection{Hopf actions on quadratic algebras}

First, we introduce some notation. Let $A$ be a quadratic algebra
generated by
$x_1,x_2,\cdots,x_n$ in degree 1, subject to the relations
$$r_w:=\sum c^{ij}_w x_i x_j=0$$
for $w=1,\cdots, m$, where $c^{ij}_w\in k$ \cite[Chapter 5]{Manin:QGNCG}.
We assume that $\{r_1,\cdots, r_m\}$ is linearly independent. We
will construct a universal Hopf algebra that coacts on $A$
(Definition \ref{defxx2.6}, Lemma \ref{lemxx2.7}).

Let $F$ be a free algebra generated by $\{y_{ij}\}_{1\leq i,j\leq n}$.
Consider a bialgebra structure on $F$ defined by
\begin{equation}
\label{E2.0.1}\tag{E2.0.1}
\Delta(y_{ij})=\sum_{s=1}^n y_{is}\otimes y_{sj} \quad \text{and} \quad
\epsilon(y_{ij})=\delta_{ij},
\end{equation}
for all $1\leq i,j \leq n$. The free
algebra $A':=k\langle x_1,\cdots,x_n\rangle$ is a right $F$-comodule
algebra with comodule structure map $\rho: A'\to A'\otimes F$ determined by
$$\rho(x_i)=\sum_{s=1}^n x_s\otimes y_{si}$$
for $1\leq i \leq n$.

\begin{remark}
\label{remxx2.1}
Note that Manin uses left coactions in \cite{Manin:QGNCG}\cite{Manin:Topics},
yet the comodule structure map $\rho$ above corresponds to
a right coaction. By symmetry, all results in \cite{Manin:QGNCG} hold
for right coactions. In fact, later in this work, we will use both
left and right coactions.
\end{remark}

Let $A=k\langle x_1,\cdots,x_n\rangle/(R)$ where $R=\oplus_{i=1}^m k
r_i\subset A_1^{\otimes 2}$. To determine a coaction of $K$ on $A$
induced by the coaction of $K$ on $A'$ above, we use the following
explicit construction. Let $V=A_1$ and $V^*$ be the $k$-linear dual
of $V$. Identify $(V^*)^{\otimes 2}$ with $(V^{\otimes 2})^*$
naturally. The {\it Koszul dual} of $A$, denoted by $A^!$, is an
algebra $k\langle V^*\rangle/(R^{\perp})$ where $R^{\perp}$ is the
subspace
$$R^{\perp}:=\{ a\in (V^*)^{\otimes 2}\mid \langle a, r_w
\rangle=0, {\text{for all $w$}}\}=\{ a\in (V^*)^{\otimes 2}\mid
\langle a, R \rangle=0\}.$$ By duality,
$$R=\{ r\in V^{\otimes 2}\mid \langle R^{\perp}, r\rangle=0\}.$$
Pick a basis for $R^{\perp}$, say $r'_{u}$ for $u=1,\cdots, n^2-m$,
and write
$$r'_u=\sum_{i,j}d^{ij}_u x^*_i x^*_j$$
for all $1\leq u\leq n^2-m$. By the definition of $R^{\perp}$, we
have that
\begin{equation}\label{E2.1.1}\tag{E2.1.1}
\sum_{i,j} d^{ij}_u c_w^{ij}=0
\end{equation}
for all $u,w$. The following lemma is well-known.

\begin{lemma} \cite[Lemmas 5.5 and 5.6]{Manin:QGNCG}
\label{lemxx2.2} Let $K$ be a bialgebra coacting on the free algebra
$A'= k \langle x_1, \dots, x_n \rangle$ with $\rho(x_i)=\sum_{s=1}^n
x_s \otimes a_{si}$ for $a_{ij}\in K$. Then the following are
equivalent.
\begin{enumerate}
\item
The coaction $\rho : A' \longrightarrow A' \otimes K$ satisfies $\rho \left( R
\right) \subseteq R \otimes K$, that is, 
$$\sum_{i,j,k,l} c_w^{ij} d^{kl}_u a_{ki}a_{lj}=0$$
for all $w,u$.

\item
The map $\rho$ induces naturally a coaction of $K$ on $A$ such that
$A$ is a right $K$-comodule algebra. \qed
\end{enumerate}
\end{lemma}

Let $I$ be the ideal of the bialgebra $F$, defined at the beginning
of this section, generated by elements $\sum_{i,j,k,l} c_w^{i,j}
d^{k,l}_u y_{ki}y_{lj}$ for all $w,u$. Now we can define the first
quantum group associated to $A$.

\begin{definition}
\label{defxx2.3}
\cite[Section 5.3]{Manin:QGNCG} The {\it quantum matrix
space associated to $A$} is defined to be ${\mathcal O}_A(M) =F/I$, which is
a bialgebra quotient of $F$.
\end{definition}

It follows that ${\mathcal O}_A(M)$
is non-trivial since \eqref{E2.1.1} gives $m(n^2-m)$ quadratic
relations between its generators $y_{i j}$. Thus,
${\mathcal O}_A(M)$ is noncommutative in general since for
$\mathcal{O}_A(M)$ to be commutative, we need $n^2(n^2-1)/2$
(quadratic, commutative) relations.
Moreover, $\mathcal{O}_A(M)$ is similar
to the quantum group $\underline{\eEnd}(A)$ defined in
\cite[Section 5.3]{Manin:QGNCG}, but
one sees that ${\mathcal O}_A(M)\cong (\underline{\eEnd}(A))^{coop}$
because $\rho$ corresponds to right coaction and left coactions
are discussed in \cite{Manin:QGNCG}. On the other hand, we will see
in the next proposition that ${\mathcal O}_A(M)$ coincides with Manin's $\uend(A^!)$.

Before we proceed, we introduce additional notation. Let
$A''$ be the free algebra generated by $(A_1)^*$. Define a coaction
of $F$ on $A''$ via
$\rho^!: A''\to F\otimes A''$, by
$$\rho^!(x^*_i)=\sum_{s=1}^n y_{is}\otimes x^*_s.$$
Namely, $A''$ is a left $F$-comodule algebra with the left
coaction determined by $\rho^!$. We provide more details on $\rho^!$
and its connection to the homological codeterminant as follows.

\begin{remark}
\label{remxx2.4}  Given a $K$-coaction
on an AS regular algebra $A$ with $K$ having a bijective antipode,
we can reinterpret the homological codeterminant via the coaction
$\rho^!$ on the Ext-algebra of $A$. Let $E$ denote the
$\Ext$-algebra $\bigoplus_{i\geq 0} \Ext^i_A(k,k)$. (If $A$ is
Koszul, then $E$ is isomorphic to $A^!$ \cite[Theorem
5.9.4]{Smith:elliptic}.) By Remark~\ref{remxx1.7}(d), $E$ is a left
$K$-comodule algebra and let $\rho^!$ be the left $K$-coaction on
$E$.  Note that $E$ is Frobenius \cite[Corollary D]{LPWZ:Koszul}.
Let $\fe$ be a nonzero element of  maximal degree of $E$. Then
$k\fe$ is a left $K$-comodule. Since $\Ext^1_A(k,k)$ can be
naturally identified with $(A_1)^*$, it is not hard to see that the
left $K$-coaction $\rho^!$ restricted on $\Ext^1_A(_Ak,_Ak)$ agrees
with the induced left $K$-coaction on $(A_1)^*$. Now by Definition
\ref{defxx1.6}(b), the homological codeterminant of the $K$-coaction
on $A$ is defined to be the element ${\sf D}\in K$ such that
$$\rho^!(\fe)={\sf D}\otimes \fe.$$
Note that ${\sf D}$ is called the {\it quantum determinant} by Manin
\cite[Section 8.2]{Manin:QGNCG}.
\end{remark}

Following \cite[Section 8.5]{Manin:QGNCG}, we have the following result.

\begin{proposition}
\label{proxx2.5} \cite[Theorem 5.10]{Manin:QGNCG} Let $K$ be a
bialgebra  (not necessarily $F$ above) coacting on the free algebra
$A' = k\langle x_1, \dots, x_n \rangle$ with $\rho(x_i)=\sum_{s=1}^n
x_s \otimes a_{si}$ for $a_{ij}\in K$. Then the following are
equivalent.
\begin{enumerate}
\item The coaction $\rho : A' \longrightarrow A' \otimes K$ satisfies $\rho \left( R
\right) \subseteq R \otimes K$, that is, 
$$\sum_{i,j,k,l} c_w^{ij} d^{kl}_u a_{ki}a_{lj}=0$$
for all $w,u$.
\item
The map $\rho$ induces naturally a right coaction of $K$ on $A$ such that
$A$ is a right $K$-comodule algebra.
\item
The map $\rho^!$ induces a natural left coaction of $K$ on $A^!$ such that
$A^!$ is a left $K$-comodule algebra.
\item
There is a bialgebra homomorphism $\phi: {\mathcal O}_A(M)\to K$
defined by $\phi(y_{ij})=a_{ij}$.
\end{enumerate}
\end{proposition}

\begin{proof} (a) $\Leftrightarrow$ (b) This is Lemma
\ref{lemxx2.2}.

(a) $\Leftrightarrow$ (c) This follows from a version of Lemma
\ref{lemxx2.2} for $(A^!, \rho^!)$.

(a) $\Leftrightarrow$ (d) This follows from the definition of
${\mathcal O}_A(M)$.
\end{proof}

For any  bialgebra $B$, Manin defines the {\it Hopf envelope}
of $B$ in \cite[Chapter~7]{Manin:QGNCG}. This is a Hopf algebra $H(B)$ with a bialgebra map $\phi: B \to H(B)$, so that for any Hopf algebra $H'$, a bialgebra map $\psi: B \to H'$ extends uniquely to a Hopf algebra map $\hat{\psi}: H(B) \to H'$ with $\psi = \hat{\psi} \circ \phi$ as bialgebra maps. Such a structure does not necessarily exist, but it does in our setting, and we use this object to
introduce the second quantum group associated to the (quadratic) algebra $A$.

\begin{definition}\cite[page 47]{Manin:QGNCG}
\label{defxx2.6} The {\it quantum general linear group} associated
to $A$ is defined to be the Hopf envelope of ${\mathcal O}_A(M)$,
and is denoted by ${\mathcal O}_{A}(GL)$. Abusing notation, we use
$y_{ij}$ to denote the image of the generators $y_{ij} $ of
$\mathcal{O}_A(M)$ in ${\mathcal O}_{A}(GL)$.
\end{definition}

We see that $\mc{O}_A(GL)$ serves as the universal Hopf algebra that
coacts on $A$ inner-faithfully.

\begin{lemma}
\label{lemxx2.7} Let $A$ be a quadratic algebra generated by
$A_1=\oplus_{i=1}^n kx_i$ and $K$ be a Hopf algebra coacting on $A$
from the right. Write $\rho_K(x_i)=\sum_{s=1}^n x_s\otimes a_{si}$
for some $a_{si}\in K$. Then the $K$-coaction is inner-faithful if and
only if,  for all $1\leq i,j\leq n$, the map $\phi: y_{ij}\to a_{ij}$
induces a surjective Hopf algebra homomorphism from ${\mathcal
O}_{A}(GL)$ to $K$.
\end{lemma}

\begin{proof} By Proposition \ref{proxx2.5}, there is a bialgebra
homomorphism $\phi:{\mathcal O}_A(M)\to K$ defined by
$\phi(y_{ij})=a_{ij}$ for all $i,j$. Since $K$ is a Hopf algebra, by
\cite[Theorem~7.3(c)]{Manin:QGNCG}, this map extends uniquely to a Hopf algebra
homomorphism $\phi: \mathcal{O}_{A}(GL)\to K$ where $\phi(y_{ij})=a_{ij}$
for all $i,j$. Let $K'$ be the image of $\phi$, which is a Hopf
subalgebra of $K$ that coacts on $A$. If $K$-coaction on $A$ is
inner-faithful, then $K'=K$, namely,
the map $\phi$ is surjective. Conversely, if $K$-coaction is not
inner-faithful, there is a Hopf subalgebra $K'\subsetneq K$ such that
$\rho(A)\subset A\otimes K'$. By definition, each $a_{ij}$ is in
$K'$. By \cite[Theorem 7.3(c)]{Manin:QGNCG}, $\phi$ maps ${\mathcal
O}_{A}(GL)$ to $K'$. Here,  $\phi$ is not surjective.
\end{proof}

\subsection{Hopf actions on non-quadratic algebras}

In order to describe Hopf actions on $N$-Koszul algebras for $N\geq3$,
we have to
consider connected graded algebras that are not quadratic. The
definition of an $N$-Koszul algebra will be reviewed in Section
\ref{secxx4}. For the rest of this section, we consider general
connected graded algebras which are not necessarily quadratic and
collect some basic facts without proofs. Let $A$ be connected graded
algebra generated by a finite dimensional graded vector space $W$
and $K$ be a Hopf algebra coacting on $A$ via
$\delta: A \rightarrow A \otimes K$ from the
right such that:
\begin{enumerate}
\item[(H1)]
each homogeneous component $A_i$ is a right $K$-comodule;
\item[(H2)]
$A$ is a right $K$-comodule algebra; and
\item[(H3)]
$K$-coaction on $W$ is inner-faithful.
\end{enumerate}

\noindent The following is the first quantum group to which we
associate to $A$.

\begin{definitionlemma}
\label{defxx2.8} \cite[Section 7.5]{Manin:QGNCG} There is
a {\it quantum general linear group}, denoted by ${\mathcal
O}_A(GL)$ (see Definition \ref{defxx2.6} when $A$ is quadratic),
coacting on $A$, via
$$\rho: A\to A\otimes {\mathcal O}_A(GL).$$
This coaction has the universal property: for every $K$-coaction
$\delta$ on  $A$, there is a unique Hopf algebra morphism $\gamma:
{\mathcal O}_A(GL)\to K$ such that $\delta=(\text{Id} \otimes \gamma)\circ \rho$. \qed
\end{definitionlemma}

Further, $\gamma$ is surjective since the $K$-coaction is inner-faithful
(see Lemma \ref{lemxx2.7}). Now to introduce the last two quantum groups
associated to $A$ in this section, suppose that $A$ is AS regular.

\begin{definition}
\label{defxx2.9} Let $A$ be an AS regular algebra. Let ${\sf D}$ be
the homological codeterminant of the ${\mathcal O}_A(GL)$-coaction
on $A$.
\begin{enumerate}
\item \cite[Section 8.5]{Manin:QGNCG}
The {\it quantum special linear group} associated to $A$, denoted by
${\mathcal O}_{A}(SL)$, is defined to be the Hopf algebra quotient
${\mathcal O}_{A}(GL)/({\sf D}-1_K)$.
\item
Let $I$ be the Hopf ideal generated by $\{S^2(a)-a\mid {\text{for
all $a$}}\in {\mathcal O}_{A}(GL)\}$. The {\it quantum $S^2$-trivial
linear group} associated to $A$, denoted by ${\mathcal
O}_{A}(GL/S^2)$, is defined to be the Hopf algebra quotient
${\mathcal O}_{A}(GL)/I$.
\end{enumerate}
\end{definition}

Both ${\mathcal O}_{A}(SL)$ and ${\mathcal O}_{A}(GL/S^2)$ are
Hopf algebras coacting on $A$ from the right such that $A$ is a
comodule algebra. The following lemma follows from the universal
property of ${\mathcal O}_{A}(GL)$.

\begin{lemma}
\label{lemxx2.10} Retain the setting of Definition \ref{defxx2.9}
and assume ${\text{(H1)-(H3)}}$.
\begin{enumerate}
\item
If the $K$-coaction has trivial homological codeterminant, then
the quotient morphism $\gamma: {\mathcal O}_A(GL)\to K$ factors
through ${\mathcal O}_A(SL)$.
\item
If $S^2=Id_K$, then the quotient morphism $\gamma:
{\mathcal O}_A(GL)\to K$ factors through ${\mathcal O}_A(GL/S^2)$. \qed
\end{enumerate}
\end{lemma}

\section{Hopf actions on Frobenius algebras}
\label{secxx3}

Let $E$ be a connected graded Frobenius algebra. The goal of this section is to study Hopf actions on $E$. We first determine the Nakayama
automorphism $\mu_E$ in terms of a basis of $E_1$; see
Equation \eqref{E3.0.4} below.

Let $l$ be the highest degree of any nonzero element in $E$, and let
${\mathfrak e}$ be a nonzero element in $E_l$. Since $E$ is
Frobenius, the multiplication of $E$ defines a non-degenerate
bilinear form
$$ E_{i}\times E_{l-i}\to k{\mathfrak e}\cong k(-l)$$
for each $0\leq i\leq l$ \cite[Lemma 3.2]{Smith:elliptic}. Pick any
basis $\{a_1,\cdots,a_n\}$ of $E_1$.  There is a unique basis of
$E_{l-1}$, say $\{b_1,\cdots,b_n\}$, such that
\begin{equation} \label{E3.0.1} \tag{E3.0.1}
a_i b_j=\delta_{ij} {\mathfrak e}
\end{equation}
for all $1\leq i,j \leq n$.  Moreover, there is a unique basis of $E_1$,
say $\{c_1,\cdots, c_n\}$, such that
\begin{equation} \label{E3.0.2} \tag{E3.0.2}
b_i c_j=\delta_{ij} {\mathfrak e}
\end{equation}
for all $1\leq i,j \leq n$. Therefore, there is non-singular matrix
$\alpha=(\alpha_{ij})_{i,j=1}^n$ such that
\begin{equation}
\label{E3.0.3}\tag{E3.0.3} c_i=\sum_{j=1}^n \alpha_{ij} a_j
\end{equation}
for all $i$.  By the definition of Nakayama automorphism, $\mu_E$
maps $a_i$ to $c_i$, or in other words
\begin{equation}
\label{E3.0.4}\tag{E3.0.4} \mu_E(a_i)=\sum_{j=1}^{n} \alpha_{ij}a_j
\end{equation}
for all $i$. Throughout this work, we refer to $\alpha = (\alpha_{ij})$ as the matrix associated to $\mu_E$.

Now let $K$ be a Hopf algebra coacting on $E$ from the left such that
$E$ is a left $K$-comodule algebra. The $K$-coaction
(via comodule structure map  $\rho$) preserves the grading of $E$. By \cite[Theorem
5.2.2]{Manin:QGNCG}, there is a canonical map $\gamma:{\mathcal O}_E(GL)\to
K$ (see Lemma \ref{lemxx2.7} when $E$ is quadratic). So if
$\{y_{ij}\}_{1\leq i,j\leq n}$ is a set of elements in $K$ such that
\begin{equation}\label{E3.0.5}
\tag{E3.0.5} \rho(a_i)=\sum_{s=1}^n y_{is}\otimes a_s
\end{equation}
for all $i$, then $\Delta(y_{ij}) = \sum_{s=1}^n y_{is} \otimes
y_{sj}$ and $\epsilon(y_{ij}) = \delta_{ij}$. 

For simplicity, denote the matrix $(y_{ij})_{n\times
n}$ by ${\mathbb Y}$. For any matrix ${\mathbb W}=(w_{ij})_{n\times n}$
in $M_n(K)$, write $S({\mathbb W})=(S(w_{ij}))_{n\times n}$.
Choose $b_i$ and $c_i$ as in (\ref{E3.0.1}, \ref{E3.0.2}) and write
$$\rho(b_i)=\sum_{s=1}^n f_{is}\otimes b_s,$$
$$\rho(c_i)=\sum_{s=1}^n g_{is}\otimes c_s$$
for elements $f_{is},g_{is}$ in $K$. Note that each $g_{ij}$ is in
the space $\sum_{s,t} ky_{st}$. If $E$ is generated in degree 1,
then each $f_{ij}$ is contained in the subalgebra generated by $\left\{ y_{s t} \right\}_{s, t}$. Similarly,
let ${\mathbb F}=(f_{ij})_{n\times n}$ and
${\mathbb G}=(g_{ij})_{n\times n}$ be the corresponding matrices.

\begin{lemma}
\label{lemxx3.1} Retain the notation above and let $\mathbb{I}$ be
the $n\times n$ identity matrix. Let ${\sf D}$ be the homological
codeterminant of the left $K$-coaction $\rho$ on $E$. (This is
applicable to $E$ in Definition \ref{defxx1.6}(b) for instance.) We
have the statements below.
\begin{enumerate}
\item
${\mathbb Y} S({\mathbb Y})={\mathbb I}=S({\mathbb Y}){\mathbb Y}.$
\item
${\mathbb G} S({\mathbb G})={\mathbb I}=S({\mathbb G}){\mathbb G}.$
\item
${\mathbb Y}{\mathbb F}^\tau={\sf D} \mathbb{I}$. As a consequence, $S({\mathbb Y})
={\mathbb F}^\tau ({\sf D}^{-1} \mathbb{I})$.
\item
$S({\mathbb F})S({\mathbb Y}^\tau)={\sf D}^{-1} \mathbb{I}$.
\item
$S({\mathbb G}) S({\mathbb F}^\tau)={\sf D}^{-1} \mathbb{I}$. As a consequence,
$S({\mathbb F}^\tau) \cdot {\sf D}\mathbb{I}={\mathbb G}$.
\item
$S^2({\mathbb Y})={\sf D}\mathbb{I} \cdot  {\mathbb G} \cdot  {\sf D}^{-1}\mathbb{I}$.
\item
${\mathbb G}=\alpha {\mathbb Y} \alpha^{-1}$ where
$\alpha=(\alpha_{ij})$.
\end{enumerate}
\end{lemma}

\begin{proof} (a) By definition $\rho(a_i)=\sum_{s=1}^n y_{is}
\otimes a_s$ for all $i$. By the coassociativity of $\rho$, we have
$$\Delta(y_{ij})=\sum_{s=1}^n y_{is}\otimes y_{sj} \quad {\text{and}}\quad
\epsilon(y_{ij})=\delta_{ij}$$ for all $i,j$.  The
assertion follows from the antipode axiom.
\smallskip

\noindent (b) This is similar to (a).
\smallskip

\noindent (c) Applying $\rho$ to the equation $\delta_{ij} {\mathfrak e}= a_i
b_j$, we have by Definition \ref{defxx1.6}(b) that
$$
\delta_{ij} {\sf D}\otimes {\mathfrak e}=\left (\sum_{s} y_{is}\otimes
a_s \right) \left(\sum_t f_{jt}\otimes b_t\right)=\sum_{s} y_{is}f_{js}\otimes
{\mathfrak e}.
$$
This implies that ${\mathbb Y}{\mathbb F}^\tau={\sf D} \mathbb{I}$. Since ${\sf D}$ is
invertible in $K$, we have that ${\mathbb Y}({\mathbb F}^\tau {\sf D}^{-1}
\mathbb{I})=\mathbb{I}$. So ${\mathbb F}^\tau {\sf D}^{-1} \mathbb{I}$ is a right inverse of
${\mathbb Y}$. By part (a), $S({\mathbb Y})$ is an (left and right) inverse of
${\mathbb Y}$. Hence, $S({\mathbb Y})={\mathbb F}^\tau {\sf D}^{-1} \mathbb{I}$.

\smallskip

\noindent (d) This follows by applying $S$ to ${\mathbb Y}{\mathbb
F}^\tau={\sf D} \mathbb{I}$ from part (c), and using the fact $S$ is
an anti-endomorphism of $K$. Note that ${\sf D}$ is necessarily
grouplike, so $S({\sf D})={\sf D}^{-1}$. 

\smallskip

\noindent (e) Applying $\rho$ to the equation $\delta_{ij} {\mathfrak e}= b_i
c_j$, we have that
$$
\delta_{ij} D\otimes {\mathfrak e}=\left(\sum_{s} f_{is}\otimes b_s\right)\left(\sum_t
g_{jt}\otimes c_t\right)=\sum_{s} f_{is}g_{js}\otimes {\mathfrak e},
$$
which implies that ${\mathbb F}{\mathbb G}^\tau={\sf D} \mathbb{I}$.
Applying $S$ we obtain that $S({\mathbb G}) S({\mathbb F}^\tau)={\sf
D}^{-1} \mathbb{I}$. 
The last assertion follows by the fact
$S({\mathbb G})$ has the inverse ${\mathbb G}$ by part (b).

\smallskip

\noindent (f) The assertion follows from parts (c,e).

\smallskip

\noindent (g) This follows from applying $\rho$ to the equations
$c_i=\sum_{j=1}^m \alpha_{ij} a_j$ and linear algebra.
\end{proof}

We now define an algebra endomorphism $\eta_{\mu_E}$ of $K$ dependent on the
Nakayama automorphism $\mu_E$ of the Frobenius algebra $E$. First choose a
basis $\left\{ a_i \right\}_{i = 1}^n$ of $E_1$ and let $\alpha \in M_n \left(
k \right)$ be the matrix of $\mu_E \mid_{E_1}$ with respect to this basis
(c.f. \eqref{E3.0.4}). Let $\mathbb{Y}= \left( y_{i j} \right)_{n\times n}$ where $\rho \left( a_i
\right) = \sum_{s = 1}^n y_{i s} \otimes a_s$ and define
\begin{equation}
\label{E3.1.1}\tag{E3.1.1}
  (\eta_{\mu_E} \left( y_{ij} \right))_{n\times n} =\eta_{\mu_E} \left( \mathbb{Y} \right) = \alpha \mathbb{Y} \alpha^{-
  1} .
\end{equation}
If the $K$-coaction on $E_1$ is inner-faithful, then the entries of
$\mathbb{Y}$ generate $K$ as a Hopf algebra. In this case, we show below that
\eqref{E3.1.1} extends to an algebra endomorphism of $K$ which is independent of the
choice of basis $\left\{ a_i \right\}_{i = 1}^n$ of $E_1$.

\begin{theorem}
\label{thmxx3.2} Assume that
\begin{enumerate}
\item[(i)]
$E$ is connected graded and Frobenius with Nakayama automorphism
$\mu_E$;
\item[(ii)]
$K$ is a Hopf algebra with antipode $S$ coacting on $E$ such that
the left $K$-coaction on $E_1$ is inner-faithful; and
\item[(iii)]
${\sf D} \in K$ is the homological codeterminant of the left $K$-coaction on $E$.
\end{enumerate}
Then \eqref{E3.1.1} determines a unique Hopf algebra endomorphism
of $K$, still denoted by $\eta_{\mu_E}$, and
\begin{equation}
\label{E3.2.1}\tag{E3.2.1}
\eta_{\mu_E}(y_{ij}) = {\sf D}^{-1}S^2(y_{ij}){\sf D}
\end{equation}
for all $i,j$. As a consequence, the antipode $S$ of $K$ is surjective.
\end{theorem}

\begin{proof}
By Lemma \ref{lemxx3.1}(f,g),
$${\sf D}^{-1}S^2({\mathbb Y}) {\sf D}={\mathbb G}=\alpha {\mathbb Y} \alpha^{-1}=
\eta_{\mu_E}({\mathbb Y})$$
which is \eqref{E3.2.1}. Let $\eta_{\sf D}$ be the conjugation by ${\sf D}$.
Then $\eta_{\sf D}\circ S^2$ is a Hopf algebra endomorphism of $K$. Equation
\eqref{E3.1.1} says that $\eta_{\mu_E}=\eta_{\sf D}\circ S^2$ when applied
to $y_{ij}$. Since $\{y_{ij}\}$ generates $K$ as a Hopf algebra,
there is a unique Hopf algebra endomorphism of $K$ which extends
$\eta_{\mu_E}$. We denote this endomorphism by $\eta_{\mu_E}$ again.
By \eqref{E3.1.1}, $\eta_{\mu_E}=\eta_{\sf D}\circ S^2$ on all of $K$.

By definition, $\eta_{\mu_E}$ is bijective when restricted to
the space spanned by $\{y_{ij}\}_{i,j}$. Therefore, $\eta_{\mu_E}$ is
surjective on $K$. Since $\eta_{\sf D}$ is an automorphism, $S^2$ is
a surjection. Therefore, $S$ is a surjection.
\end{proof}

\begin{theorem}
\label{thmxx3.3} Suppose that
\begin{enumerate}
\item[(i)]
$E$ is connected graded and Frobenius;
\item[(ii)]
the map $\mu_E\mid E_1$ is scalar multiplication
\item[(iii)]
$E$ is a left $K$-comodule algebra such that the $K$-coaction on $E_1$ is
inner-faithful; and
\item[(iv)]
 the $K$-coaction on $E$ has trivial homological
codeterminant. \end{enumerate} Then $S^2=Id$. If in addition,
$\ch k$ does not divide $\dim H$, then $H=K^{\circ}$ is
semisimple.
\end{theorem}

\begin{proof} By hypothesis (ii) and \eqref{E3.1.1},
$\eta_{\mu_E}$ is the identity on
$y_{ij}\in K$ for all $i,j$ for $\{y_{ij}\}$ in \eqref{E3.0.5}. By
hypothesis (iii) and Lemma \ref{lemxx1.3}(b), $\eta_{\mu_E}$ is the
identity on $K$. By Theorem~\ref{thmxx3.2}, we have that
$S^2(y)={\sf D} y {\sf D}^{-1}$ for all $y\in K$. By hypothesis
(iv), ${\sf D}=1$, so $S^2=Id$. Since $\ch k$ does not divide $\dim H$, we have that $H$ is semisimple.
\end{proof}

\begin{remark}
\label{remxx3.4}
The automorphism $\eta_{\mu_E}$ is defined by using a basis of
$\left\{ a_i \right\}_{i = 1}^n$ of $E_1$. Theorem~\ref{thmxx3.2}
shows that the construction $\eta_{\mu_E}$ is independent of the choice of
such a basis.
\end{remark}

\section{Proof of the main results}
\label{secxx4}

In this section, we prove Theorems \ref{thmxx0.1}, \ref{thmxx0.4},
and \ref{thmxx0.6}. First we introduce and provide preliminary
results on {\it $N$-Koszul algebras}. The concept of an $N$-Koszul
algebra was introduced by Berger in \cite{Berger:Koszulity}. If $N=2$,
then an $N$-Koszul algebra is just a usual Koszul algebra. The
following is the definition for general $N$. Let $V$ be a
finite-dimensional $k$-vector space and let $N$ be an integer larger
than 1. Let $R$ be a subspace of $V^{\otimes n}$ and let $A=k\langle
V\rangle /(R)$.  The algebra $A$ is called {\it $N$-Koszul} if the
left trivial $A$-module $k$  has a free resolution of the form
$$\cdots \to A(-s(i))^{d_i}\to \cdots \to A(-s(2))^{d_2}\to
A(-s(1))^{d_1}\to A\to k\to 0$$
where $s(2j) = Nj$ and $s(2j + 1) = Nj + 1$; see \cite[Definition 2.10]{Berger:Koszulity}
and the discussion in \cite[Section 2]{Berger:Koszulity}. We need the following
property about $N$-Koszul algebras to prove some of our main results.

\begin{lemma}
\label{lemxx4.1}\cite[Theorem 6.3]{BergerMarconnet} Let $A$ be an $N$-Koszul AS
regular algebra of global dimension $d$. Let $\mu_A$ be the Nakayama
automorphism of $A$ (as defined in the introduction) and let $\mu_E$ be the Nakayama automorphism of
its $\Ext$-algebra $E$ (as defined in \eqref{E1.5.1}). Then
$\mu_{E}\mid_{E_1}=(-1)^{d+1} (\mu_A\mid_{A_1})^*$. \qed
\end{lemma}

Let $A$ be an $N$-Koszul AS regular algebra with $\{x_1,\cdots,x_n\}$ a basis of $A_1$. 
Retain the notation in the lemma above, and let $\mathbb{M}$ be the matrix $(m_{ij})_{n\times n}$
such that
$$\mu_A(x_i)=\sum_{j=1}^n m_{ij} x_j$$
for all $i$. If $\{x_i^*\}_{i=1}^n$ is the basis of $E_1:=
\Ext^1_A(k,k)$, which is dual to the basis $\{x_i\}_{i=1}^n$, then
by Lemma \ref{lemxx4.1},
$$\mu_{E}(x_i^*)=\sum_{j=1}^n (-1)^{d+1} m_{ji} x_j^*$$
for all $i$. Note that if we use the notation introduced in \eqref{E3.0.4}, then
we have  that $\alpha=(-1)^{d+1}\mathbb{M}^\tau$ is the matrix associated to the Nakayama automorphism of $E$.

Now consider the automorphism $\eta_{\mu_A^{\tau}}$ on $K$ defined by
conjugating by the transpose of the corresponding matrix $\mathbb{M}$ of $\mu_A$:
\begin{equation*} \tag{E4.1.1} \label{E4.1.1}
\eta_{\mu_A^\tau}(y_{ij})
=\sum_{s,t=1}^n m_{si}y_{st}n_{jt},
\end{equation*}
for all $1\leq i,j\leq n$.
Here, $(n_{ij})_{n\times n}=\mathbb{M}^{-1}$.

We are now ready to prove Theorem \ref{thmxx0.1}.

\begin{proof}[Proof of Theorem \ref{thmxx0.1}]
By \cite[Section 7.5]{Manin:QGNCG} (also by Remark
\ref{remxx1.7}(d)), $K$ coacts on $E$ from the left such that $E$ is
a left $K$-comodule algebra. 
If the right $K$-coaction on $A_1$ is given by 
$\rho(x_i)=\sum_{j=1}^n x_j\otimes y_{ji}$, 
then the left $K$-coaction on $E_1=\Ext^1_A(k,k)$ is given by 
$\rho(x^*_i)=\sum_{j=1}^n y_{ij}\otimes x^*_i$. Since the
$K$-coaction on $A$ is inner-faithful and $A$ is generated by $A_1$,
by Lemma \ref{lemxx1.3}(c), the $K$-coaction on $A_1$ is
inner-faithful. Hence the $K$-coaction on $E_1$ is inner-faithful.
Therefore by Theorem~\ref{thmxx3.2}, $\eta_{\mu_E} = \eta_{\sf D}\circ S^2$. Here,
${\sf D}$ is the homological codeterminant of the left $K$-coaction on $E$.
Also by Definition \ref{defxx1.6}(b), ${\sf D}$ is the homological codeterminant of the right $K$-coaction
on $A$. 
By Lemma \ref{lemxx4.1}, we get that $\mu_{E}\mid_{E_1}=(-1)^{d+1}
(\mu_A\mid_{A_1})^*$. Now $\eta_{\mu_E}=\eta_{\mu_A^\tau}$ holds by
(\ref{E3.1.1}), (E4.1.1), and the equation $\alpha = (-1)^{d+1} \mathbb{M}^{\tau}$ above. The assertion follows.
\end{proof}

\begin{remark}
\label{remxx4.2} We conjecture that a version of Lemma
\ref{lemxx4.1} holds for general AS regular algebras generated in
degree 1. If this were true, then Theorem \ref{thmxx0.1} holds for
any AS regular algebras that are generated in degree 1, particularly
for those that are not necessarily $N$-Koszul.
\end{remark}

Next we aim to prove Theorem \ref{thmxx0.4}. Let $\{p_{ij}\mid i<j\}$ be a
set of parameters in $k^{\times}$. For this set, let $k_{p_{ij}}[x_1,\cdots,x_n]$
be the {\it skew polynomial ring} generated by $x_1,\cdots,x_n$ and
subject to the relations $$x_jx_i=p_{ij}x_ix_j$$ for all $i<j$. Here, we take
$p_{ii}=1$ and $p_{ji}=p_{ij}^{-1}$ for all $i<j$.

\begin{theorem}
\label{thmxx4.3} Suppose $k$ is algebraically closed. Let $A$ be the
skew polynomial ring $k_{p_{ij}}[x_1,\cdots,x_n]$. Suppose a
finite dimensional Hopf algebra $H$ acts on $A$ inner-faithfully.
Assume that, for
each pair $(i,j)$ of distinct indices, $\prod_{a=1}^n
(p_{ia}p_{aj})$ is not a $2l$-th root of unity, where $l:=\dim H$. Then $H$ is a group
algebra.
\end{theorem}

\begin{proof} It is known that $A$ is a Koszul algebra and
$A^!$ is Frobenius. We have a basis $\{ x_1^*, \dots, x_n^*\}$ of
$A_1^!$, subject to the relations
$$x_j^* x_i^* + p_{ij}^{-1} x_i^* x_j^* \text{~for all~} i <j
\quad \quad  \text{and} \quad \quad (x_i^*)^2 = 0 \text{~for all~} i.$$
Let $\{ \hat{x}_1^*, \dots, \hat{x}_n^* \}$ be a basis of $A_{n-1}^!$
where $\hat{x}_i^*:=x_1^* \cdots x_{i-1}^* x_{i+1}^* \cdots x_n^*$.
Moreover, take $\mathfrak{e}$ to be  $x_1^{\ast} \cdots
x_n^{\ast} \in A^!_n$. Now
$$x_i^* \hat{x}_j^* = \delta_{ij} (-p_{1i}^{-1})
\cdots (-p_{i-1,i}^{-1}) \mathfrak{e} = \delta_{ij} (-p_{1j}^{-1})
\cdots (-p_{j-1,j}^{-1}) \mathfrak{e}.$$
Let $b_j$ denote $(-p_{1j}) \cdots (-p_{j-1,j}) \hat{x}_j^*$
so that $$x^*_i b_j = \delta_{ij} \mathfrak{e}.$$
Compare this to Equation \eqref{E3.0.1}.

On the other hand,
\[
\begin{array}{ll}
b_i x_j^* &= (-p_{1i}) \cdots (-p_{i-1,i}) \hat{x}_i^* x_j^*\\
                &= \delta_{ij} \displaystyle  \prod_{a=1}^{i-1}
(-p_{ai}) \cdot p_{ii} \cdot \displaystyle \prod_{a=i+1}^n
(-p_{ai}) \cdot \mathfrak{e}.
\end{array}
\]
Let $c_j$ denote $\prod_{a=1}^n (-1)^{n-1} p_{ja}x_j^*$ so that
$$b_i c_j = \delta_{ij} \mathfrak{e}.$$
Compare this to Equation \eqref{E3.0.2}.

By Equations (\ref{E3.0.3}, \ref{E3.0.4}), we have that the Nakayama
automorphism $\mu_{A^!}$ is defined by
$$\mu_{A^!}(x_i^*) = (-1)^{n+1} \prod_{a=1}^n p_{ia} x_i^*.$$

Hence the matrix associated to $\mu_{A^!}$ is the diagonal matrix
$\alpha$ with
$(i,i)$-entry being $(-1)^{n+1} \prod_{a=1}^n p_{ia}$.
Now let $\{y_{ij}\}$ be elements of $K$ satisfying Equation
\eqref{E3.0.5} and Remark \ref{remxx2.4}, that is to say
$$\rho^!(x_i^*) = \sum_{s=1}^n y_{is} \otimes x_s^*.$$ Put
$\mathbb{Y}:=(y_{ij})$. Then the $(i,j)$ entry of
$\alpha \mathbb{Y} \alpha^{-1}$ is $q_{ij} y_{ij}$, where
$q_{ij} = \prod_{a=1}^n p_{ia}p_{aj}$.

Now for $K=H^\circ$ coacting on $A$,  we aim to show that $K$ is
commutative. Recall that $l=\dim K=\dim
H<\infty$. Let $\eta_{\sf D}$ be the conjugation automorphism by the
homological codeterminant ${\sf D}$ of the $K$-coaction. Then Theorem
\ref{thmxx0.1} implies that $\eta_{\sf D} \circ S^2$ sends $y_{ij}$ to
$q_{ij} y_{ij}$. Since $S^2({\sf D})={\sf D}$, $\eta_{\sf D}$ commutes
with $S^2$.
Since ${\sf D}\in K$, the order $o({\sf D})$ of ${\sf D}$ divides $l$
by the Nichols-Zoeller theorem
\cite[Theorem~3.1.5]{Montgomery}. By Radford's theorem, (see
\cite[page~209]{RadfordSchneider}), we have that $o(S^2)\mid 2l$ .
So $m:=o(\eta_{\sf D}\circ
S^2)$ divides $2l$. For any $i\neq j$,
$$y_{ij }=(\eta_{\sf D} \circ S^2)^m(y_{ij})=q_{ij}^{m} y_{ij}.$$
Since we assume that $o(q_{ij})$ does not divide $2l$, we have that
$q_{ij}^{m}\neq 1$. Thus $y_{ij}=0$ for all $i\neq j$. Thus $K$ is
generated by grouplike elements $y_{ii}$. 
 Since $y_{ij}=0$
for all $i\neq j$, $\rho^!(x_i^*)=y_{ii}\otimes  x_i^*$ for all $i$.
Then the relation $x_j^*x_i^*=-p_{ji}x_i^*x_j^*$ implies that
$y_{ii}$ and $y_{jj}$ commute. Therefore $K$ is commutative,
and $H=K^\circ$ is a group algebra as desired \cite[Theorem 2.3.1]{Montgomery}.
\end{proof}

\begin{proof}[Proof of Theorem \ref{thmxx0.4}]
 Here we take $p_{ij}=p$
for all $i<j$. It is direct to check that $q_{ij}:= \prod_{a=1}^n (p_{ia} p_{aj}) =p^{2(j-i)}$
is not a root of unity for all $i\neq j$, since $p$ is not a root of unity. Hence Theorem
\ref{thmxx0.4} follows from Theorem \ref{thmxx4.3}.
\end{proof}

Finally, we prove Theorem \ref{thmxx0.6} below.

\begin{proof}[Proof of Theorem \ref{thmxx0.6}]
Retain the notation as above and assume that the hypotheses of
Theorem \ref{thmxx0.6} hold. When $A$ is $r$-Nakayama of global
dimension $d$, the matrix $\mathbb{M}$  corresponding to the Nakayama automorphism
$\mu_A$ is $(-1)^{d+1}r \mathbb{I}$ by Lemma~\ref{lemxx4.1}. In this
case, Theorem \ref{thmxx0.1} states that
$${\sf D}^{-1}S^2({\mathbb Y}){\sf D}= (\mathbb{M}^{\tau})^{-1} \mathbb{Y} \mathbb{M}^{\tau} = {\mathbb Y}.$$
Since we also assume that the $K$-coaction has trivial homological
codeterminant, we have ${\sf D}=1$. Thus
$S^2({\mathbb Y})={\mathbb Y}$
or $S^2$ is the identity on the set $\{y_{ij}\}$. By Lemma~\ref{lemxx1.3}(b), $K$ is generated by $\{y_{ij}\}$ as a Hopf
algebra. Thus $S^2$ is the identity on $K$.
 Since $\ch k$ does not divide $\dim K$, we have that $K$ and $K^{\circ}$ are
semisimple.
\end{proof}

\begin{remark}
A natural question is when the skew polynomial $A:=k_{p_{ij}}[x_1,\cdots,x_n]$ is
$r$-Nakayama. By the proof of Theorem \ref{thmxx0.4} and Lemma
\ref{lemxx4.1}, $\mu_A(x_i)=(\prod_{a=1}^n p_{ia})x_i$ for each
$i=1,\cdots,n$. Hence, $A$ is $r$-Nakayama if and only if
$\prod_{a=1}^n p_{ia}=r$ for all $i$. Note that $\prod_{i=1}^n
(\prod_{a=1}^n p_{ia})=1$. This implies that $r^n=1$. So, $A$ is
$r$-Nakayama if and only if both $r^n=1$ and $p_{1i}=r^{-1}\prod_{a=2}^n
p_{ia}$ for every $i$. Here, $\{p_{ij}\}_{2 \leq i<j \leq n}$ are independent
variables. A special case occurs when $p_{ij}=-1$ for all $1\leq i<j\leq n$; 
in this case, $A$ is $(-1)^{n-1}$-Nakayama.
\end{remark}

\section{Examples}
\label{secxx5} In this section, we provide some explicit
 examples of the main results and of  some of the quantum
 groups discussed in Section~\ref{secxx2}. In particular, we take
$A$ to be an AS regular algebra of global dimension 2.
We show that if such an $A$ is non-PI, then there are
no non-trivial Hopf algebra actions on $A$ (c.f. Theorem \ref{thmxx5.10}). 
 We also prove
Proposition \ref{proxx0.7} here. 

\subsection{For the skew polynomial ring $A=A_p:=k_p[x_1,x_2]$}
Recall that $$A_p := k_p[x_1,x_2]=k\langle x_1,x_2\rangle /(x_2x_1-px_1x_2).$$
Suppose a Hopf algebra $K$ coacts on $A_p$ with
$$\rho(x_i)= x_1\otimes y_{1i} +  x_2\otimes y_{2i}.$$
for some $y_{si}\in K$ for $i=1,2$.

\begin{lemma}
\label{lemxx5.1} If $K$ coacts on $A_p$, then we get the following relations for $\{y_{ij}\}$:
$$\begin{aligned}
y_{12}y_{11}- p y_{11}y_{12} &=0, \\
y_{22}y_{21}- p y_{21}y_{22}&=0, \\
y_{22}y_{11}-p y_{21}y_{12}&= {\sf D}, \\
y_{11}y_{22}-p^{-1} y_{12}y_{21} &={\sf D},\\
\end{aligned}$$
where ${\sf D}$ is the homological codeterminant of the $K$-coaction
on $A_p$.
\end{lemma}

\begin{proof} Since $K$ coacts on $A_p$, the coaction $\rho$ maps the relation
of $A_p$ to zero. This means that $\rho(x_2x_1-px_1x_2)\in k\langle x_1, x_2 \rangle \otimes K$
generates a 1-dimensional right $K$-comodule. Hence, there is a grouplike
element $g\in K$ such that
\[
\begin{array}{ll}
\rho(x_2x_1-px_1x_2)&=(x_2x_1-px_1x_2)\otimes g\\
&=x_1x_2\otimes (-p g)+x_2x_1\otimes g.
\end{array}
\]
By a direct computation, we have that
$$\begin{aligned}
\rho(x_2x_1-px_1x_2)&=x_1^2\otimes y_{12}y_{11}+x_1x_2\otimes
y_{12}y_{21} +x_2x_1\otimes y_{22}y_{11}+x_2^2\otimes y_{22}y_{21}\\
&\quad-p(x_1^2\otimes y_{11}y_{12}+x_1x_2\otimes y_{11}y_{22}
+x_2x_1\otimes y_{21}y_{12}+x_2^2\otimes y_{21}y_{22})\\
&= x_1^2\otimes (y_{12}y_{11}-py_{11}y_{12})+x_1x_2\otimes
(y_{12}y_{21}-py_{11}y_{22})\\
&\quad +x_2x_1\otimes (y_{22}y_{11}-py_{21}y_{12})+x_2^2\otimes
(y_{22}y_{21}-py_{21}y_{22}).
\end{aligned}
$$
By comparing the coefficients of $x_ix_j$'s, we have that
$$\begin{aligned}
y_{12}y_{11}-p y_{11}y_{12} &=0,\\
y_{22}y_{21}- p y_{21}y_{22}&=0, \\
y_{22}y_{11}-p y_{21}y_{12}&= g,\\
y_{11}y_{22}-p^{-1} y_{12}y_{21} &=g.
\end{aligned}
$$

Using the notation at the beginning at Section~\ref{secxx3}, we have that $$A^!
= E = \frac{k \langle x_1^*, x_2^* \rangle}{(x_2^* x_1^* + p^{-1}
x_1^* x_2^*, ~(x_1^*)^2, ~(x_2^*)^2)}.$$ Pick $\mathfrak{e} = x_1^*
x_2^* \in E$. Finally, applying $\rho^!$ (discussed in Remark
\ref{remxx2.4}) to $\mathfrak{e}$, we have that the homological
codeterminant ${\sf D}$ is given by
$$\rho^!(\mathfrak{e}) = {\sf D} \otimes \mathfrak{e} =
(y_{11}\otimes x_1^* +y_{12}\otimes x_2^*)(y_{21}\otimes x_1^* +y_{22}\otimes x_2^*).$$
Thus ${\sf D}=g =y_{11}y_{22}-p^{-1} y_{12}y_{21}$. The assertion follows.
\end{proof}

The following is an application of Lemma \ref{lemxx3.1}(c).

\begin{lemma}
\label{lemxx5.2} Retain the above notation.
$$\begin{aligned}
S(y_{11})&=y_{22} {\sf D}^{-1},\\
S(y_{12})&=-p y_{12} {\sf D}^{-1},\\
S(y_{21})&=-p^{-1} y_{21} {\sf D}^{-1},\\
S(y_{22})&=y_{11} {\sf D}^{-1}.
\end{aligned}$$
\end{lemma}

\begin{proof}
Let $a_1 := x_1^*$, $a_2:=x_2^*$ be a basis of $E_1$. Now we need
a basis $\{b_1, b_2\}$ of $E_1$ so that $a_i b_j = \delta_{ij}
\mathfrak{e}$. Here, $b_1 = x_2^*=a_2$ and $b_2 = -px_1^*=-pa_1$.
Thus
\[
\begin{array}{rll} \notag
\rho^!(b_1) &=  y_{21} \otimes x_1^* + y_{22} \otimes x_2^*
&= -p^{-1} y_{21} \otimes b_2 +y_{22} \otimes b_1, \text{~~and}\\
\notag
\rho^!(b_2) &= -py_{11} \otimes x_1^*  - py_{12} \otimes x_2^*
&= y_{11} \otimes b_2 - py_{12} \otimes b_1 .
\end{array}
\]
Therefore, the result follows from Lemma \ref{lemxx3.1}(c).
\end{proof}

Towards the  proof of Proposition \ref{proxx0.7}, let us assume that
$K$ is semisimple. We now get the following relations among $\{y_{ij}\}$.

\begin{lemma}
\label{lemxx5.3}
Suppose that $S^2=Id$. Then $K$ has relations
$$\begin{aligned}
y_{21}y_{11}&=p^{-1}y_{11}y_{21},\\
y_{22}y_{12}&=p^{-1}y_{12}y_{22},\\
py_{21}y_{12}&=p^{-1}y_{12}y_{21},\\
y_{22}y_{11}&=y_{11}y_{22}.
\end{aligned}
$$
\end{lemma}

\begin{proof} We have from Lemma \ref{lemxx3.1}(a) that
\[
\begin{array}{rl}
y_{22}{\sf D}^{-1}y_{12}& =p y_{12}{\sf D}^{-1}y_{22}, \\
y_{21}{\sf D}^{-1}y_{11}&= p y_{11}{\sf D}^{-1}y_{21}, \\
y_{22}{\sf D}^{-1}y_{11}-p y_{12}{\sf D}^{-1}y_{21}&= 1, \\
y_{11}{\sf D}^{-1}y_{22}-p^{-1} y_{21}{\sf D}^{-1}y_{12} &=1.
\end{array}
\]
Applying $S^2$ to $y_{ij}$, we obtain that
$$\begin{aligned}
S^2(y_{11})&={\sf D}y_{11} {\sf D}^{-1},\\
S^2(y_{12})&= p^2 {\sf D} y_{12}{\sf D}^{-1},\\
S^2(y_{21})&= p^{-2} {\sf D}y_{21}{\sf D}^{-1},\\
S^2(y_{22})&= {\sf D}y_{22}{\sf D}^{-1}.
\end{aligned}$$
Since $S^2 = Id$, we have that
$$\begin{aligned}
{\sf D}y_{11} &= y_{11} {\sf D},\\
{\sf D} y_{12}&= p^{-2}y_{12} {\sf D},\\
{\sf D}y_{21} &= p^2 y_{21} {\sf D},\\
{\sf D}y_{22} &=y_{22} {\sf D}.
\end{aligned}$$
Hence, the first and last set of equations, along with Lemma \ref{lemxx5.1}, yield the result.
\end{proof}

Now we compute the quantum groups $\mathcal{O}_{A_p}(SL)$ and
$\mathcal{O}_{A_p}(GL/S^2)$ associated to $A_p$ (Definition
\ref{defxx2.9}). To do this, we consider a two-parameter family of
quantum $GL_n$, denoted by $GL_{\alpha,\beta}(n)$; such quantum
groups were defined by Takeuchi \cite[Section
2]{Takeuchi:two-parameter}. It follows from the definition that the
standard quantum group ${\mathcal O}_q(GL_n(k))$ is equal to Takeuchi's
$GL_{q,q}(n)$.

\begin{proposition}
\label{proxx5.4} Suppose that a finite dimensional Hopf algebra $K$
(with antipode $S$) coacts on the skew polynomial ring $A_p$ with
$S^2=Id_K$. Then $K$ is a Hopf quotient of Takeuchi's quantum group
$GL_{p,p^{-1}}(2)$. As a consequence, the quantum group ${\mathcal
O}_{A_p}(GL/S^2)=GL_{p,p^{-1}}(2)$.
\end{proposition}

\begin{proof}
The relations in Lemma \ref{lemxx5.1} and \ref{lemxx5.3} combined give
all relations of quantum group $GL_{p,p^{-1}}(2)$ as defined in
\cite[Section 2]{Takeuchi:two-parameter}. Hence, the assertion follows.
\end{proof}

\begin{proposition} \label{proxx5.5} Suppose that the $K$-coaction
on $A_p$ has trivial homological codeterminant, then $K$ is a Hopf
quotient of $\mathcal{O}_p(SL_2(k))$. Thus,  ${\mathcal O}_{A_p}(SL)=\mathcal{O}_p(SL_2(k))$.
\end{proposition}

\begin{proof} 
Since the $K$-coaction has trivial cohomological
determinant, ${\sf D}=1_K$. Now the relations in Lemma
\ref{lemxx5.1}, the relations obtained from applying the antipode (computed in Lemma~\ref{lemxx5.2}) to these relations, together with ${\sf D}=1_K$
provide a complete set of relations for $\mathcal{O}_p(SL_2(k))$.
The assertion follows.
\end{proof}

Now we are ready to prove Proposition \ref{proxx0.7}.

\begin{proof}[Proof of Proposition \ref{proxx0.7}]
Let $K=H^{\circ}$. Then $K$ coacts on $k[x_1,x_2]$ inner-faithfully.

Since $H$ is semisimple, so is $K$. Consequently, $S_K^{2}=Id_K$. By
Proposition \ref{proxx5.4}, $K$ is a Hopf quotient of
$GL_{p,p^{-1}}(2)$, where $p=1$ in this case. Since $GL_{1,1}(2)={\mathcal
O}(GL_2)$ is commutative, so is $K$. Therefore $H$ is cocommutative.
Since $H$ is finite dimensional over an algebraically closed field,
$H$ is a group algebra as desired.
\end{proof}

\subsection{For the Jordan plane $A=A_J:=k_J[x_1,x_2]$}
In this subsection, we provide computations similar to those in the
previous subsection for the algebra
 $$A_J := k_J[x_1,x_2]=k\langle x_1,x_2\rangle /(x_2x_1-x_1x_2-x_1^2).$$
Suppose that $A_J$ is a $K$-comodule algebra with comodule
structure map\\ $\rho: A_J\to A_J\otimes K$ defined by
$$\rho(x_i)= x_1\otimes a_{1i} +  x_2\otimes a_{2i}$$
for some $a_{si} \in K$ with $i=1,2$.

\begin{lemma}
\label{lemxx5.6} Retain the notation above. Let ${\sf D}$ be the
homological codeterminant of the $K$-coaction on $A_J$. Then we have
that
$$\begin{aligned}
a_{12}a_{11}-a_{11}a_{12}-a_{11}^2&= -{\sf D},  \\
a_{12}a_{21}-a_{11}a_{22}-a_{11}a_{21}&=-{\sf D},\\
a_{22}a_{11}-a_{21}a_{12}-a_{21}a_{11}&={\sf D}, \\
a_{22}a_{21}-a_{21}a_{22}-a_{21}^2&=0.
\end{aligned}$$
\end{lemma}

\begin{proof} As before in Subsection~\ref{secxx5}.1, use the coefficients of $\rho(x_2 x_1 - x_1 x_2 -x_1^2)$.
\end{proof}

The next lemma is an application of Lemma \ref{lemxx3.1}(c).

\begin{lemma}
\label{lemxx5.7} Retain the notation above. Then
$$\begin{aligned}
S(a_{11})&= (a_{22}+a_{21}) {\sf D}^{-1},\\
S(a_{12})&= (-a_{11}-a_{12}+a_{21}+a_{22}) {\sf D}^{-1},\\
S(a_{21})&= -a_{21} {\sf D}^{-1},\\
S(a_{22})&= (a_{11}-a_{21}) {\sf D}^{-1}.
\end{aligned}$$
\vspace{-.35in}

\qed
\end{lemma}

Moreover, the following lemma follows from Lemma \ref{lemxx5.7}.

\begin{lemma}
\label{lemxx5.8} Retain the notation above. Then
$$\begin{aligned}
S^2(a_{11})&= {\sf D} (a_{11}-2a_{21}) {\sf D}^{-1},\\
S^2(a_{12})&= {\sf D}(a_{12}+2a_{11}-2a_{22}-4a_{21}) {\sf D}^{-1},\\
S^2(a_{21})&= {\sf D} a_{21} {\sf D}^{-1},\\
S^2(a_{22})&= {\sf D}(a_{22}+2a_{21}) {\sf D}^{-1}.
\end{aligned}$$
\vspace{-.36in}

\qed
\end{lemma}

Now we will see that $K$ must be a group algebra as in the skew polynomial case.

\begin{proposition}
\label{proxx5.9} Suppose that $K$ coacts on $A_J$
inner-faithfully with $\dim K<\infty$, with $K$ not necessarily semisimple. If $\ch k =0$,
then $K=k G$ where $G = C_n$ for some $n$.
\end{proposition}

\begin{proof}
Using the notation introduced in Section~\ref{secxx1}, we have that $K$ is generated by
elements ${\sf D}, {\sf D}^{-1}$ and $\{a_{ij}\}_{i,j=1,2}$.
Let $\eta_{\sf D}$ be the conjugation automorphism by ${\sf D}$, namely,
$\eta_{\sf D}: f\mapsto  {\sf D}^{-1}f {\sf D}$ for all $f\in K$.
Lemma \ref{lemxx5.8}
implies that $\eta_{\sf D} \circ S^2$ sends $a_{11}$ to $a_{11}-2a_{21}$.

On the other hand, since $S^2({\sf D})={\sf D}$, $\eta_{\sf D}$ commutes with $S^2$. Since the subalgebra
generated by $D$ is a Hopf subalgebra of $K$, we have that the order $o({\sf D})$ of ${\sf D}$ divides $l:=\dim K$. By
Radford's theorem \cite[page 209]{RadfordSchneider}, $o(S^2)\mid 2l$. So
$m:=o(\eta_{\sf D}\circ S^2)$ divides $2l$. By a computation,
$$a_{11}=(\eta_D \circ S^2)^m(a_{11})=a_{11}-2m a_{21}.$$
Since we assume that $\ch k =0$, we have $2m\neq 0$ in
$k$. Thus $a_{21}=0$. An argument using $(\eta_D \circ S^2)^m(a_{12}) = a_{12}$ shows that
$a_{11}-a_{22}=0$.
Thus $K$ is generated by two  elements $a_{11}$
and $a_{12}$ with
$\Delta(a_{12})=a_{11}\otimes a_{12}+a_{12}\otimes a_{11}.$
Then $a_{11}^{-1}a_{12}$ is a primitive element.
 Since $K$ is finite
dimensional and $\ch k =0$, there is no non-trivial
primitive element. Therefore $a_{12}=0$ and $K$ is generated by
$a_{11}$. As a consequence, $K=k C_n$ for some $n$.
\end{proof}

\begin{theorem}
\label{thmxx5.10} Let $k$ be algebraically closed  and $A$ be a
noetherian non-PI AS regular algebra of dimension two that is
generated in degree 1. Suppose that a finite dimensional Hopf algebra $H$
acts on $A$ inner-faithfully. Then $H$ is a group algebra.
\end{theorem}

\begin{proof} Since $k$ is algebraically closed, every AS
regular algebra of global dimension two is isomorphic
to either $k_p[x_1,x_2]$ or $k_J[x_1,x_2]$.
\smallskip

\noindent \underline{Case 1}: $A=k_p[x_1,x_2]$. Since $A$ is not PI,
$p$ is not a root of
unity. The assertion follows from Theorem \ref{thmxx0.4}.
\smallskip

\noindent \underline{Case 2}: $A=k_J[x_1,x_2]$. Since $A$ is not PI,
$\ch k=0$. The assertion follows from Proposition~\ref{proxx5.9}.
\end{proof}

Now that we have studied Hopf actions on AS regular algebras of
global dimension 2, where the Gelfand-Kirillov dimension of $H$
is 0. We finish with a question.

\begin{question}
\label{quexx5.11} Let $p$ be not a root of unity and
$A=k_p[x_1,x_2]$. Is there a noncommutative Hopf algebra $K$ of
$\GKdim$ 1 coacting on $A$ inner-faithfully?
\end{question}

\subsection*{Acknowledgments}
We thank the anonymous referee for valuable comments that improved the exposition of this work. C. Walton and J.J. Zhang were supported by the US National
Science Foundation: NSF grants DMS-1102548 and DMS-0855743, respectively.

\bibliography{HopfAS2trivHdet_biblio}

\end{document}